\documentclass[twoside,12pt]{article}
\usepackage[all, 2cell, dvips]{xy}
\usepackage{latexsym,amsmath, amsfonts, graphics, epsf, epic}
\usepackage{amsthm}
\usepackage{amssymb}
\usepackage{amsopn}
\usepackage{amscd}
\usepackage{color}

\theoremstyle{plain}
\newtheorem{thm}{Theorem}[section]
\newtheorem{cor}[thm]{Corollary}

\newtheorem{prop}[thm]{Proposition}

\theoremstyle{definition}

\usepackage{amssymb}

\def\sg{\sigma}

\def\lm{\lambda}

\def\dim{\mbox{\rm dim }}

\def\l.l.o.{\it l.l.o}

\def\f{\varphi}

\def\C{{\cal C}}

\def\chiup{\raise 2pt\hbox{$\chi$}}

\title{Lie superalgebras and some characters of $S_n$}

\author{Amitai Regev\\
Department of Mathematics,\\
The Weizmann Institute of science,\\
Rehovot, Israel\\
e-mail: amitai.regev at weizmann.ac.il\\
}

\begin{document}

\maketitle

{\bf Abstract}. We prove a formula for $S_n$ characters which are indexed by the partitions in the
$(k,\ell)$ hook. The proof applies a combinatorial part of the theory of Lie superalgebras~\cite{berele.regev}.

\section{Introduction}

$S_n$ is the $n$-th symmetric group. The irreducible $S_n$ characters are denoted $\chi^\lm$, where
$\lm$ is a partition of $n$, denoted $\lm\vdash n$.
The number of non-zero parts of a partition $\lm$ is denoted $\ell(\lm)$. The conjugacy class of $S_n$ corresponding to
a partition $\mu\vdash n$
-- via disjoint cycle decomposition --
is denoted $\C_\mu\subseteq S_n$. Since characters are constant on conjugacy classes,
if $\sg\in\C_\mu$ we write $\chi^\lm(\sg)=\chi^\lm(\mu)$.

\subsection{Main results}
Studying character tables of $S_n$, one observes the following intriguing phenomena.

\begin{prop}\label{main.1}
Let $\mu=(\mu_1,\ldots,\mu_r)\vdash n$ where $\mu_r>0$, and
let $\chi_n$ be the following $S_n$-character:
$$
\chi_n=\sum_{i=0}^{n-1} \chi^{(n-i,1^i)}.
$$

\medskip
1. ~If some $\mu_j$ is even then $\chi_n(\mu)=0$.

\medskip
2. ~If all $\mu_j$s are odd then $\chi_n(\mu)=2^{\ell(\mu)-1}$.
\end{prop}

Proposition~\ref{main.1} is a special case of Theorem~\ref{identity.1} below, which
gives a more general  identity of $S_n$-characters.
We prove that general identity by applying some combinatorial parts of
the theory of Lie superalgebras~\cite{berele.regev}.

\medskip
It seems to be of some
 interest to find proofs of these results which only use $S_n$ character theory.

\subsection{Preliminaries}

\medskip
Given the integers $k,\ell\ge 0$, let $H(k,\ell;n)$ denote the partitions of $n$ in the
$(k,\ell)$-hook:
\begin{eqnarray}\label{hook.1}
H(k,\ell;n)=\{\lm=(\lm_1,\lm_2,\ldots)\vdash n\mid \lm_{k+1}\le \ell\}.
\end{eqnarray}
Also let $H'(k,\ell;n)\subseteq H(k,\ell;n)$ denote the subset of the partitions
containing the $k\times\ell$ rectangle:
\begin{eqnarray}\label{hook.2}
H'(k,\ell;n)=\{\lm\in H(k,\ell;n)\mid \lm_k\ge\ell\}.
\end{eqnarray}
Recall that a tableau is called {\it semi-standard} if it is weakly increasing in rows and
strictly in columns; if its entries are from the set $\{1,2,\ldots,k\}$ then it is called
a $k$ tableau. Of course a  $k$ tableau is also a $k+1$ tableau, etc. Let $\lm\vdash n$, then
$s_k(\lm)$ denotes the number of $k$ semi standard tableaux of shape $\lm$. For formulas
for $s_k(\lm)$, see for example~\cite{macdonald},~\cite{stanley}.
Generalization to $(k,\ell)$ semi standard tableaux is given in~\cite{berele.regev}. Here, the number of
$(k,\ell)$ semi standard tableaux of shape $\lm$ is denoted by $s_{k,\ell}(\lm)$. Formulas for
calculating $s_{k,\ell}(\lm)$ -- when $\lm\in H'(k,\ell;n)$ -- are given in~\cite[Section 6]{berele.regev}.

\medskip
We briefly review  the part of the theory of Lie super algebras needed here, see~\cite{berele.regev}.
We first review the relevant parts of the classical Schur-Weyl theory.

\subsubsection{The classical Schur-Weyl theory}

Let $\dim V=k$  and let $V^{\otimes n}=V\otimes\cdots \otimes V$
$n$ times.
Define $\f_{k,n}:FS_n\to End(V^{\otimes n})$ via linearity and
\begin{eqnarray}\label{left.action1}
\sg\in S_n, \qquad \f_{k,n}(\sg)(v_1\otimes\cdots \otimes v_n)=
v_{\sg^{-1}(1)}\otimes\cdots \otimes v_{\sg^{-1}(n)}.
\end{eqnarray}
Then $\f_{k,n}$ is an $S_n$-representation, and we let $\chi_{\f_{k,n}}$
denote its $S_n$-character.
Via the representation $\f_{k,n}$, $V^{\otimes n}$ is a left $FS_n$ module. The decomposition of
$V^{\otimes n}$ into irreducibles $FS_n$ modules implies the following classical theorem.

\begin{thm}\label{classical.1}
Recall that $s_k(\lm)$ is the number of $k$-semi-standard tableaux of shape $\lm$.
Then
$$
\chi_{\f_{k,n}}=\sum_{\lm\in H(k,0;n)}s_k(\lm)\cdot\chi^\lm.
$$
\end{thm}

\subsubsection{The super analogue ~\cite{berele.regev}}

Here $V=V_0\oplus V_1$, $\dim V_0=k,$ $\dim V_1=\ell$, so $\dim V=k+\ell$. The super analogue of $\f_{k,n}$ is the map
$$
\f^*_{(k,\ell),n}:FS_n\to End((V_0\oplus V_1)^{\otimes n}),
$$
which is an $S_n$ representation,
 see~\cite{berele.regev} for details. Let $\sg\in S_n$. The main feature of $\f^*_{(k,\ell),n}(\sg)$ is, that when it
commutes an elements from $V_0$ with any element of $V$, this produces a plus sign; and when it
commutes two elements from $V_1$, this produces a minus sign.
We denote by $\chi_{\f^*_{(k,\ell),n}}$ the corresponding $S_n$ character.

\section{The main results}

Recall that
 $s_{k,\ell}(\lm)$ is the number of $k,\ell$ semi-standard tableaux
of shape $\lm$~\cite{berele.regev}. The following theorem, which is the super analogue of
Theorem~\ref{classical.1}, follows from the
analogue decomposition of $(V_0\oplus V_1)^{\otimes n}$ into irreducile left $FS_n$ modules~\cite{berele.regev}.
\begin{thm}\label{super.2}
$$
\chi_{\f^*_{(k,\ell),n}}=\sum_{\lm\in H(k,\ell;n)}s_{k,\ell}(\lm)\chi^\lm.
$$
\end{thm}

Given $\mu\vdash n,$
we now calculate $\chi_{\f^*_{(k,\ell),n}}(\mu)$ as the trace of the matrix of ${\f^*_{(k,\ell),n}}(\mu)$.
\begin{thm}\label{main.3}
Let $\mu=(\mu_1,\ldots,\mu_r)\vdash n$ with $\mu_r>0$, then
$$
\chi_{\f^*_{(k,\ell),n}}(\mu)=\prod_{j=1}^r(k+(-1)^{\mu_j+1}\ell).
$$

\end{thm}

\begin{proof}
To calculate $\chi_{\f^*_{(k,\ell),n}}(\mu),$ let $\sg\in\C_\mu$ be the following permutation, given by its
disjoint cycle decomposition:
$$
\sg=(1,2,\ldots \mu_1)(\mu_1+1,\ldots,\mu_1+\mu_2)\cdots
$$
and we calculate $\chi_{\f^*_{(k,\ell),n}}(\sg).$ To do that, we choose a basis of $(V_0\oplus V_1)^{\otimes n}$,
calculate the matrix $M_\sg$ of $\f^*_{(k,\ell),n}(\sg)$ with respect to that basis, then calculate the trace
of $M_\sg$.

\medskip
So let $t_1,\ldots,t_k\in V_0$ and
$u_1,\ldots ,u_\ell\in V_1$ be bases, denote
$(t_1,\ldots,t_k,u_1,\ldots ,u_\ell)=(v_1,\ldots,v_{k+\ell})$, so $v_1,\ldots,v_{k+\ell}$ is
a (homogeneous) basis of $V$. Then
$$
\{v_{i_1}\otimes\cdots\otimes v_{i_n}\mid 1\le i_j\le k+\ell\}
$$
is a basis of $(V_0\oplus V_1)^{\otimes n}$. Consider the $j$-th cycle $(r,r+1,\ldots, s)$ in the disjoint cycle decomposition of $\sg$. It is of length $\mu_j=s-r+1$. Also consider a basis element $\bar v=v_{i_1}\otimes\cdots\otimes v_{i_n}$.
Corresponding to that cycle $(r,\ldots, s)$ we have the factor $v_{i_r}\otimes\cdots\otimes v_{i_s}$
of $\bar v$. If  $v_{i_r}\ne v_{i_q}$ for some $r+1\le q\le s$ then $\f^*_{(k,\ell),n}(\bar v)\ne \bar v$
hence $\bar v$ contributes $0$ to the trace of $M_\sg$. Hence we can assume that $v_{i_r}=\cdots = v_{i_s}$.

\medskip
There is the possibility that $v_{i_r}\in\{t_1,\ldots,t_k\}$ and
the possibility that $v_{i_r}\in\{u_1,\ldots,u_\ell\}.$
From the first possibility we get a contribution of $k$ to $tr(M_\sg)$
(since there are $k$ possible values $t_1,\ldots,t_k$ for $v_{i_r}$).
Similarly, from the second possibility we get a contribution of $(-1)^{\mu_j+1}\ell$ (the sign $(-1)^{\mu_j+1}$ is the sign of the cycle $(r,\ldots,s)$).
Together, the  $j$-th cycle $(r,\ldots, s)$
of $\sg$ contributes  to $tr(M_\sg)$ the factor $k+(-1)^{\mu_j+1}\ell$. This completes the proof.
\end{proof}
As an obvious consequence of Theorems~\ref{super.2} and~\ref{main.3} we have
\begin{thm}\label{identity.1}
Let $\mu=(\mu_1,\ldots,\mu_r)\vdash n$ with $\mu_r>0$, then
$$
\sum_{\lm\in H(k,\ell;n)}s_{k,\ell}(\lm)\cdot\chi^\lm(\mu)=\prod_{j=1}^r(k+(-1)^{\mu_j+1}\ell).
$$
\end{thm}

\section{Special cases}
\subsection{The case  $\ell=0$}
When $\ell=0$, $\f^*_{(k,0),n}=\f_{k,n}$ and Theorem~\ref{main.3} implies that
$\chi_{\f_{k,n}(\mu)}=k^{\ell(\mu)}$, which is a classical (known) result.
\subsection{The case $k=\ell=1$: the proof of Proposition~\ref{main.1}}

When $k=\ell=1$, by~\cite[Theorem 6.24]{berele.regev} we get  $s_{1,1}(\lm)=2$ for any $\lm\in H(1,1;n)$.
Note that $H(1,1;n)=\{(n-i,1^i)\mid i=0,1,\ldots,n-1\}$. Let $\mu=(\mu_1,\ldots,\mu_r)\vdash n$ with $\mu_r>0$, then by Theorem~\ref{identity.1}
$$
\sum_{i=0}^{n-1} 2\cdot\chi^{(n-i,1^i)}(\mu)=\prod_{j=1}^r(1+(-1)^{\mu_j+1})
$$

1. ~If some $\mu_j$ is even then the right hand side is zero, hence $\sum_{i=0}^{n-1} \chi^{(n-i,1^i)}(\mu)=0.$

\medskip
2. ~And if all $\mu_j$s are odd then  the right hand side $=2^{\ell(\mu)}$, namely $\sum_{i=0}^{n-1} \chi^{(n-i,1^i)}(\mu)=2^{\ell(\mu)-1}$. 

\medskip
This completes the proof of Proposition~\ref{main.1}.

\subsection{The case $k=2$ and $\ell=1$}
Here we can prove
\begin{cor}\label{identity.2.1}
Recall that $H'(2,1;n)\subseteq H(2,1;n)$ denote the partitions $\lm=(\lm_1,\lm_2,\ldots)\in H(2,1;n)$ with $\lm_2>0$,
and let  $\mu=(\mu_1,\ldots,\mu_r)\vdash n$ with $\mu_r>0$. Then
$$
\sum_{\lm\in H'(2,1;n)}(\lm_1-\lm_2+1)\cdot \chi^\lm(\mu)=
\frac{1}{4}\left(\prod_{j=1}^r\big(2+(-1)^{\mu_j+1}\big)-(2n+1) \right)
$$
\end{cor}
\begin{proof}
The only $\lm\in H(2,1;n)$ and $\lm\not \in H'(2,1;n)$ is $\lm=(n)$, and in that case it follows from~\cite[Definition 2.1]{berele.regev}
that $s_{2,1}(n)=2n+1$. When $\lm\in H'(2,1;n)$,~\cite[Theorem 6.24]{berele.regev} implies that $s_{2,1}(\lm)=4(\lm_1-\lm_2+1)$.  The proof now follows by applying Theorem~\ref{identity.1}.

\end{proof}

\end{document}